\documentclass[11pt]{article}
\usepackage{enumerate}
\usepackage{amssymb,a4wide,latexsym,makeidx,epsfig,fleqn}
\usepackage{amsthm}
\usepackage{amsmath}
\usepackage{enumerate}
\usepackage{extarrows}
\usepackage{graphicx}
\usepackage{subfigure}
\usepackage{float}
\usepackage[justification=centering]{caption}
\newtheorem{theorem}{Theorem}[section]

\newtheorem{definition}[theorem]{Definition}
\newtheorem{lemma}[theorem]{Lemma}

\newtheorem{corollary}[theorem]{Corollary}

\begin{document}
\textwidth 150mm \textheight 225mm

\title{ On the $A_\alpha$ spectral radius and $A_\alpha$ energy of digraphs
\thanks{This work is supported by the National Natural Science Foundation of China (No. 12001434).}}
\author{{Weige Xi}\\
{\small College of Science, Northwest A\&F University, Yangling, Shaanxi 712100, China.}\\
{\small E-mail:  xiyanxwg@163.com}\\}

\date{}
\maketitle
\begin{center}
\begin{minipage}{120mm}
\vskip 0.3cm
\begin{center}
{\small {\bf Abstract}}
\end{center}
{\small  Let $G$ be a digraph with adjacency matrix $A(G)$ and outdegrees diagonal matrix $D(G)$. For any real $\alpha\in[0,1]$, the $A_\alpha$ matrix $A_\alpha(G)$ of a digraph $G$ is defined as
$A_\alpha(G)=\alpha D(G)+(1-\alpha)A(G)$. The eigenvalue of $A_\alpha(G)$ with the largest modulus is called the $A_\alpha$ spectral radius of $G$. In this paper, we first give some upper bounds for the
$A_\alpha$ spectral radius of a digraph and we also characterize the extremal digraphs attaining these bounds. Moreover, we define the $A_\alpha$ energy of a digraph
$G$ as $E^{A_\alpha}(G)=\sum\limits_{i=1}^n(\lambda^\alpha_i(G))^2$, where $n$ is the number of vertices
and $\lambda^\alpha_i(G)$ $(i=1,2,\ldots,n)$ are the eigenvalues of $A_\alpha(G)$. We obtain a formula for $E^{A_\alpha}(G)$, and give a lower and upper bounds for $E^{A_\alpha}(G)$ and characterize the extremal digraphs that attain the lower and upper bounds. Finally, we characterize the extremal digraphs with maximum and minimum $A_\alpha$ energy among all directed trees
and unicyclic digraphs, respectively.

}
\end{minipage}
\end{center}

\section{Introduction }
\label{sec:ch6-introduction}

Let $G=(V(G),E(G))$ be a digraph with vertex set $V(G)$ and arc set $E(G)$. If $(v_i,v_j)\in E(G)$ is an arc from vertices $v_i$ to $v_j$,
we call $v_j$ the head and $v_i$ the tail of $(v_i,v_j)$, respectively. For the digraph $G$, its underlying undirected graph is an undirected graph
obtained from $G$ by removing the direction of each arc and replacing the parallel edges with a single edge. An arc $(v_i,v_j)$ is symmetric if both
$(v_i,v_j)$ and $(v_j,v_i)$ are arcs in the digraph $G$. An arc with identical head and tail is called a loop. In this paper, we only consider
the digraphs with no loops and no two of its arcs have the same head and same tail. However, we allow the digraphs to have symmetric arcs.

For a digraph $G$, a directed path of length $k\geq1$ in $G$ is a finite non-null
sequence $v_1e_1v_2e_2\ldots e_kv_{k+1}$, whose terms are alternately vertices and arcs, such that, for $i=1,2,\ldots,k$,
the arc $e_i$ has head $v_{i+1}$ and tail $v_i$, and the vertices $v_1,v_2,\ldots,v_{k+1}$ are distinct. If $v_{k+1}=v_{1}$, then we call the sequence $v_1e_1v_2e_2\ldots e_kv_{k+1}$ a directed cycle of length $k\geq2$, it is often represented simply by its vertex sequence $v_1v_2\ldots v_{k}v_1$, and we often denoted it by $\overrightarrow{C_k}$. A digraph $G$ is connected if its underlying undirected graph is connected.  A digraph $G$ is strongly
connected if for any pair of vertices $v_i,v_j\in V(G)$, there exists a directed path from $v_i$ to $v_j$. A digraph is symmetric digraph if each arc is symmetric arc.
A digraph is called complete digraph if its underlying undirected graph is a complete undirected graph. A digraph $G$ is called symmetric complete digraph if for any two distinct vertices $v_i,v_j\in V(G)$, there are arcs $(v_i,v_j)\in E(G)$ and $(v_j,v_i)\in E(G)$. The complement $\bar{G}$ of $G$ is a digraph in which for any two distinct vertices $v_i, v_j$,
$(v_i,v_j)\in E(\bar{G})$ if and only if $(v_i,v_j)\notin E(G)$. For other terminology and notations of digraphs, we refer the readers to see \cite{BoMu,BG}.

For a digraph $G$, the outdegree $d_i^+$ of vertex $v_i$ in $G$ is the number of arcs with tail $v_i$,  and the indegree $d_i^-$ of $v_i$ in $G$ is the number of arcs with head $v_i$. We denote the minimum and maximum indegrees and outdegrees in $G$ by $\delta^-(G)$,  $\Delta^-(G)$, $\delta^+(G)$, $\Delta^+(G)$, respectively. Let $T_i^+=\sum\limits_{(v_i,v_j)\in E(G)}d_j^+$ denote the 2-outdegree of vertex $v_i$, $m_i^+=\frac{T_i^+}{d_i^+}$ denote the average 2-outdegree of $v_i$. A digraph is outdegree regular if each of its vertex has the same outdegree. A digraph is $k$-regular if each of its vertex has the outdegree $k$ and indegree $k$. A bipartite digraph is one whose vertex set can be partitioned into two subsets $V_1$ and $V_2$, such that each arc has one end in $V_1$ and one end in $V_2$, such a partition $(V_1,V_2)$ is called a bipartition of the digraph. A bipartite digraph with partition $(V_1,V_2)$ is outdegree semiregular if any vertex in $V_1$ has the same outdegree and any vertex in $V_2$ has the same outdegree.


For a digraph $G$ with vertex set $V(G)=\{v_1,v_2,\ldots,v_n\}$ and arc set $E(G)$,  the adjacency matrix $A(G)=(a_{ij})$ of $G$ is an $n\times n$ matrix, where $a_{ij}=1$ if $(v_i,v_j)\in E(G)$ and 0 otherwise. Let $D(G)=\textrm{diag}(d^+_1,d^+_2,\ldots,d^+_n)$ be the diagonal matrix with outdegrees of vertices of $G$. We call $Q(G)=D(G)+A(G)$, the signless Laplacian matrix of $G$.
In 2019, J.P. Liu et al. \cite{LWCL} defined the $A_\alpha$ matrix of $G$ as
$$A_\alpha(G)=\alpha D(G)+(1-\alpha) A(G), \ \ \ \alpha\in[0,1].$$
It is clear that $A_\alpha(G)=A(G)$ if $\alpha=0$, $2A_\alpha(G)=Q(G)$ if $\alpha=\frac{1}{2}$, and $A_\alpha(G)=D(G)$ if $\alpha=1$. From this it follows that
the matrix $A_\alpha(G)$ reduces to merging the adjacent spectral and signless Laplacian spectral theories. The eigenvalues of $A_\alpha(G)$ are called the $A_\alpha$ eigenvalues of the digraph $G$, and we denoted them by $\lambda^\alpha_1(G), \lambda^\alpha_2(G), \ldots,\lambda^\alpha_n(G)$,
or simply by $\lambda^\alpha_1, \lambda^\alpha_2, \ldots,\lambda^\alpha_n$ if it is clear from the context. The eigenvalue of $A_\alpha(G)$ with the largest modulus is called the $A_\alpha$ spectral radius of $G$, denoted by $\lambda_\alpha(G)=\lambda^\alpha_1(G)$. If $G$ is a strongly connected digraph, then it follows from the Perron Frobenius Theorem \cite{HJ} that $\lambda_\alpha(G)$ is an eigenvalue of $A_\alpha(G)$, and there is a unique positive unit eigenvector corresponding to $\lambda_\alpha(G)$. Since $\alpha=1$, $A_\alpha(G)=D(G)$, which makes no sense. Therefore, we assume that $0\leq \alpha<1$ in the rest of this paper.

The adjacent spectral radius and signless Laplacian spectral radius of digraphs had received a lot of attention from researchers. For some papers
and related results in these directions see \cite{BB,DL,HoYo,LD,LSWY,XiWa1,XSW} and the references therein. However, there is no much known about the
$A_\alpha$ spectral radius of digraphs. Recently, Xi et al. \cite{XiWa} determined the digraphs which attain the maximum (or minimum) $A_\alpha$ spectral radius among all strongly connected digraphs with given parameters such as girth, clique number, vertex connectivity and arc connectivity. Xi and Wang \cite{XiW} gave some lower bounds on
$\Delta^+(G)-\lambda_\alpha(G)$ for strongly connected irregular digraphs $G$ with given maximum outdegree and some
other parameters. Ganie and Baghipur \cite{GM} obtained some lower and upper bounds on the $A_\alpha$ spectral radius of digraphs
in terms of the number of arcs, the number of closed walks and the vertex outdegrees. They also characterized the extremal digraphs attaining these bounds.

In 2008, Pe\~{n}a and Rada \cite{PR} defined the energy of a digraph as
$$E(G)=\sum\limits_{i=1}^n|Re(\lambda_i)|,$$
where $\lambda_1,\lambda_2,\ldots,\lambda_n$ are the eigenvalues of $A(G)$ and $Re(\lambda_i)$
denotes the real part of $\lambda_i$, $1 \leq i \leq n$.
For some results on the energy of digraphs, see \cite{AyBa,MJ,Ra,YW}. Recently, the Laplacian energy and signless Laplacian energy
of a digraph has attached attention. In 2010, Perera and Mizoguchi \cite{PM} defined
the Laplacian energy of a digraph $G$ as
$$LE(G)=\sum\limits_{i=1}^n \mu_i^2,$$
where $\mu_i$ are all the eigenvalues of the Laplacian matrix $L(G)=D(G)- A(G)$. They also obtained a formula for $LE(G)$
and presented lower and upper bounds for $LE(G)$ in term of the number of vertices. In 2016, Qi et al. \cite{QFL} presented a unified formula for
$LE(G)$ for both simple and symmetric digraphs. They also obtained lower and upper bounds for the Laplacian energy of digraphs and also
characterized the extremal digraphs which attain the lower and upper bounds. In 2020, Yang and Wang \cite{YaW}
determined the digraphs which have the maximal and minimal Laplacian energy among all
the directed trees, unicyclic digraphs and bicyclic digraphs with $n$ vertices, respectively. In this paper, we will study the $A_\alpha$ energy of digraphs.

In this paper, we first give some upper bounds on $A_\alpha$ spectral radius of digraphs,
and we also characterize the extremal digraphs attaining these bounds. Moreover, we define the $A_\alpha$ energy of a digraph
$G$ as $E^{A_\alpha}(G)=\sum\limits_{i=1}^n(\lambda^\alpha_i(G))^2$, where $n$ is the number of vertices
and $\lambda^\alpha_i(G)$ $(i=1,2,\ldots,n)$ are the eigenvalues of $A_\alpha(G)$.
We obtain a formula for $E^{A_\alpha}(G)$, and give a lower and upper bounds for $E^{A_\alpha}(G)$ and characterize the extremal digraphs that attain the
lower and upper bounds. Finally, we characterize the extremal digraph with maximum and minimum $A_\alpha$ energy among all directed trees
and unicyclic digraphs with $n$ vertices, respectively.

\section{Bounds for the $A_\alpha$ spectral radius of strongly connected digraphs}

In this section, we will give some upper bounds
on the $A_\alpha$ spectral radius of digraphs. The following observation can be found in \cite{XiWa}.

\noindent\begin{lemma}\label{le:1}(\cite{XiWa})
Let $G$ be a strongly connected digraph with the maximum outdegree $\Delta^+$. Then
$\lambda_\alpha(G)>\alpha\Delta^+$.
\end{lemma}

\begin{theorem}\label{thm:2} \ Let $G=(V(G), E(G))$ be a strongly
connected digraph with $n$ vertices. Let $\Delta^+$ be the maximum outdegree and  $\Delta^-$ be the maximum indegree of $G$. Then
$$ \lambda_\alpha(G)\leq \alpha \Delta^{+} +(1-\alpha)\sqrt{\Delta^{+}\Delta^{-}},$$
 and the equality holds if and only if $G$ is a regular digraph.
\end{theorem}

\begin{proof} Let
${\bf X}=(x_1,x_2,\ldots,x_n)^{T}$ be the positive unit eigenvector of $A_\alpha(G)$ corresponding to the eigenvalue $\lambda_\alpha(G)$.
Then we have
$$ \lambda_\alpha(G)x_i=\alpha d_i^+x_i+(1-\alpha)\sum\limits_{(v_i,v_j)\in E(G)}x_j.$$
Furthermore, we obtain
\begin{align*}
(\lambda_\alpha(G)-\alpha d_i^+)^2x_i^2&=(1-\alpha)^2(\sum\limits_{(v_i,v_j)\in E(G)}x_j)^2\\
&\leq(1-\alpha)^2\sum\limits_{(v_i,v_j)\in E(G)}1^2\sum\limits_{(v_i,v_j)\in E(G)}x_j^2\\
&=(1-\alpha)^2d_i^+\sum\limits_{(v_i,v_j)\in E(G)}x_j^2.
\end{align*}
Thus
\begin{align*}
 \sum\limits_{i=1}^n\left[(\lambda_\alpha(G)-\alpha d_i^+)^2x_i^2\right]&=(1-\alpha)^2 \sum\limits_{i=1}^n\left(d_i^+\sum\limits_{(v_i,v_j)\in E(G)}x_j^2\right)\\
&\leq(1-\alpha)^2 \Delta^+ \sum\limits_{i=1}^n\sum\limits_{(v_i,v_j)\in E(G)}x_j^2\\
&=(1-\alpha)^2 \Delta^+\sum\limits_{i=1}^nd_i^-x_i^2\\
&\leq(1-\alpha)^2 \Delta^+\Delta^-.
\end{align*}
On the other hand, by Lemma \ref{le:1}, $\lambda_\alpha(G)-\alpha\Delta^+>0$, we have
$$ \sum\limits_{i=1}^n\left[(\lambda_\alpha(G)-\alpha d_i^+)^2x_i^2\right]\geq(\lambda_\alpha(G)-\alpha \Delta^+)^2\sum\limits_{i=1}^nx_i^2=(\lambda_\alpha(G)-\alpha \Delta^+)^2.$$
Therefore, we get
$$(\lambda_\alpha(G)-\alpha \Delta^+)^2 \leq\sum\limits_{i=1}^n\left[(\lambda_\alpha(G)-\alpha d_i^+)^2x_i^2\right]\leq(1-\alpha)^2 \Delta^+\Delta^-,$$
which means
$$\lambda_\alpha(G)-\alpha \Delta^+\leq(1-\alpha)\sqrt{\Delta^+\Delta^-},$$
that is
$$\lambda_\alpha(G)\leq\alpha \Delta^++(1-\alpha)\sqrt{\Delta^+\Delta^-}.$$
If the equality holds, then all inequalities in the above argument must be equalities. Hence, we get
$x_1=x_2=\cdots=x_n$, $d_i^+=\Delta^+$ for all $1\leq i\leq n$, and $d_i^-=\Delta^-$ for all $1\leq i\leq n$. Furthermore,
$n\Delta^+=\sum\limits_{i=1}^n d_i^+=m=\sum\limits_{i=1}^n d_i^-=n\Delta^-$, which means $G$ is a regular digraph.

For converse, if $G$ is a $k$-regular digraph, then $d_i^+=d_i^-=k$ for all $1\leq i\leq n$ and $\lambda_\alpha(G)=k$.
Furthermore, $\alpha \Delta^++(1-\alpha)\sqrt{\Delta^+\Delta^-}=k$. Thus $\lambda_\alpha(G)=\alpha \Delta^++(1-\alpha)\sqrt{\Delta^+\Delta^-}.$
\end{proof}

\begin{theorem}\label{thm:3} \ Let $G=(V(G), E(G))$ be a strongly
connected digraph with $n$ vertices and $m$ arcs. Let $\Delta^+$ be the maximum outdegree,  $\Delta^-$ be the maximum indegree, $\delta^+$ be the minimum outdegree of $G$. Then
$$ \lambda_\alpha(G)\leq \alpha \Delta^{+} +(1-\alpha)\sqrt{m-\delta^+(n-\Delta^-)},$$
 and the equality holds if and only if $G$ is a regular digraph.
\end{theorem}

\begin{proof}  Let
${\bf X}=(x_1,x_2,\ldots,x_n)^{T}$ be the positive unit eigenvector of $A_\alpha(G)$ corresponding to the eigenvalue $\lambda_\alpha(G)$.
Similar to the proof of Theorem \ref{thm:2}, we have
\begin{align*}
(\lambda_\alpha(G)-\alpha d_i^+)^2x_i^2&\leq(1-\alpha)^2d_i^+\sum\limits_{(v_i,v_j)\in E(G)}x_j^2 \\
&=(1-\alpha)^2d_i^+\left(1-\sum\limits_{(v_i,v_j)\notin E(G)}x_j^2\right).
\end{align*}
Thus
\begin{align*}
 \sum\limits_{i=1}^n\left[(\lambda_\alpha(G)-\alpha d_i^+)^2x_i^2\right]&\leq(1-\alpha)^2 \sum\limits_{i=1}^nd_i^+- (1-\alpha)^2 \sum\limits_{i=1}^n\left(d_i^+\sum\limits_{(v_i,v_j)\notin E(G)}x_j^2\right)\\
&=(1-\alpha)^2 \left(m- \sum\limits_{i=1}^n\left(d_i^+\sum\limits_{(v_i,v_j)\notin E(G)}x_j^2\right)\right).
\end{align*}
On the other hand,
\begin{align*}
\sum\limits_{i=1}^n\left(d_i^+\sum\limits_{(v_i,v_j)\notin E(G)}x_j^2\right)&=\sum\limits_{i=1}^nd_i^+x_i^2+\sum\limits_{i=1}^n \left(d_i^+\sum\limits_{(v_i,v_j)\notin E(G), v_j\neq v_i}x_j^2\right)\\
&\geq \delta^+\sum\limits_{i=1}^nx_i^2+\delta^+\sum\limits_{i=1}^n\sum\limits_{(v_i,v_j)\notin E(G), v_j\neq v_i}x_j^2\\
&= \delta^++\delta^+\sum\limits_{i=1}^n(n-1-d_i^-)x_i^2\\
&=\delta^+\sum\limits_{i=1}^n(n-d_i^-)x_i^2\\
&\geq\delta^+\sum\limits_{i=1}^n(n-\Delta^-)x_i^2\\
&=\delta^+(n-\Delta^-).
\end{align*}
Therefore
$$(\lambda_\alpha(G)-\alpha \Delta^+)^2\leq \sum\limits_{i=1}^n\left[(\lambda_\alpha(G)-\alpha d_i^+)^2x_i^2\right] \leq(1-\alpha)^2\left (m-\delta^+(n-\Delta^-)\right),$$
which means
$$\lambda_\alpha(G)-\alpha \Delta^+ \leq(1-\alpha)\sqrt{m-\delta^+(n-\Delta^-)},$$
that is
$$\lambda_\alpha(G) \leq\alpha \Delta^++(1-\alpha)\sqrt{m-\delta^+(n-\Delta^-)}.$$
If the equality holds, then all inequalities in the above argument must be equalities. Hence, we get
$x_1=x_2=\cdots=x_n$, $d_i^+=\Delta^+=\delta^+$ for all $1\leq i\leq n$, and $d_i^-=\Delta^-$ for all $1\leq i\leq n$,
which means $G$ is a regular digraph.

For converse, if $G$ is a $k$-regular digraph, then $d_i^+=d_i^-=k$ for all $1\leq i\leq n$ and $\lambda_\alpha(G)=k$.
Furthermore, $\alpha \Delta^++(1-\alpha)\sqrt{m-\delta^+(n-\Delta^-)}=k$. Thus $\lambda_\alpha(G)=\alpha \Delta^++(1-\alpha)\sqrt{m-\delta^+(n-\Delta^-)}.$
\end{proof}

\begin{theorem}\label{thm:4} \ Let $G=(V(G), E(G))$ be a strongly
connected digraph with $n$ vertices. Then
$$\lambda_\alpha(G)\leq\max\left\{\frac{\alpha(d_i^++d_j^+)+\sqrt{\alpha^2(d_i^+-d_j^+)^2+4(1-\alpha)^2m_i^+m_j^+}}{2} : (v_i, v_j)\in E(G)\right\},$$
and if $G$ is outdegree regular or outdegree semiregular digraph, then the equality holds.
\end{theorem}

\begin{proof} Let ${\bf X}=(x_1,x_2,\ldots,x_n)^{T}$ be an positive eigenvector
corresponding to the eigenvalue $\lambda_\alpha(G)$ of $D^{-1}A_\alpha(G)D$, where
$D$ is the diagonal matrix with outdegrees of vertices of $G$. We assume that one eigencomponent $x_1$ is equal to 1 and
the other eigencomponents are less than or equal to 1, that is $x_1 = 1$ and $0 < x_k \leq 1$, for all $1 \leq k \leq n$.
Let $x_2=\max\{x_k: (v_1, v_k\in E(G)\}$. From
$(D^{-1}A_\alpha(G)D){\bf X} = \lambda_\alpha(G){\bf X}$, we have
$$\lambda_\alpha(G)x_1 = \alpha d_1^{+}x_1 +(1-\alpha) \sum\limits_{(v_1, v_k)\in E(G)}\frac{d_k^+}{d_1^+}x_k.$$
That is
$$\lambda_\alpha(G)-\alpha d_1^{+}\leq(1-\alpha)m_1^+x_2.$$
And
$$\lambda_\alpha(G)x_2 = \alpha d_2^{+}x_2 +(1-\alpha) \sum\limits_{(v_2, v_k)\in E(G)}\frac{d_k^+}{d_2^+}x_k.$$
That is
$$(\lambda_\alpha(G)-\alpha d_2^{+})x_2\leq(1-\alpha)m_2^+.$$
By Lemma \ref{le:1}, we have $\lambda_\alpha(G)>\alpha \Delta^{+}$, where  $\Delta^+$ is the maximum outdegree of $G$. Then
$$(\lambda_\alpha(G)-\alpha d_2^{+})(\lambda_\alpha(G)-\alpha d_1^{+})\leq(1-\alpha)^2m_1m_2^+,$$
which means
$$\lambda_\alpha(G)\leq\frac{\alpha(d_1^++d_2^+)+\sqrt{\alpha^2(d_1^+-d_2^+)^2+4(1-\alpha)^2m_1^+m_2^+}}{2} .$$
Hence,
$$\lambda_\alpha(G)\leq\max\left\{\frac{\alpha(d_i^++d_j^+)+\sqrt{\alpha^2(d_i^+-d_j^+)^2+4(1-\alpha)^2m_i^+m_j^+}}{2} : (v_i, v_j)\in E(G)\right\}.$$
Furthermore, if $G$ is outdegree regular or outdegree semiregular digraph, we can easily get the equality holds.


\end{proof}

\begin{theorem}\label{thm:1} \ Let $G=(V(G), E(G))$ be a strongly
connected digraph with $n$ vertices and $b_i\in \mathbb{R}^{+}$ ($1\leq i\leq n$).
Then
\begin{equation}\label{eq:ca} \lambda_\alpha(G)\leq \max\left\{\alpha d_i^{+} +(1-\alpha)\sqrt{\sum\limits_{(v_j,v_i) \in E(G)}s_j^+} : v_i \in
V(G)\right\},
\end{equation}
 where $s_i^+= \frac{\sum\limits_{(v_i,v_j) \in E(G)}b_j^2 }{b_i^2}$, and if the equality holds then
$\alpha d_i^{+} +(1-\alpha)\sqrt{\sum\limits_{(v_j,v_i) \in E(G)}s_j^+}$  ($1\leq i\leq n$) is a constant.
\end{theorem}

\begin{proof} Let ${\bf X}=(x_1,x_2,\ldots,x_n)^{T}$ be an positive eigenvector
corresponding to the eigenvalue $\lambda_\alpha(G)$ of $B^{-1}A_\alpha(G)B$, where
$B=\textrm{diag}(b_1,b_2,\ldots,b_n)$.
From
$(B^{-1}A_\alpha(G)B){\bf X} = \lambda_\alpha(G){\bf X}$, we have
$$\lambda_\alpha(G)x_k = \alpha d_k^{+}x_k +(1-\alpha) \sum\limits_{(v_k, v_h)\in E(G)}\frac{b_h}{b_k}x_h.$$
That is
$$(\lambda_\alpha(G)-\alpha d_k^{+}) x_k=(1-\alpha) \sum\limits_{(v_k, v_h)\in E(G)}\frac{b_h}{b_k}x_h.$$
By Cauchy-Schwartz inequality, we get
\begin{align*}
[(\lambda_\alpha(G)-\alpha d_k^{+}) x_k]^2&\leq(1-\alpha)^2 \sum\limits_{(v_k, v_h)\in E(G)}\frac{b_h^2}{b_k^2} \sum\limits_{(v_k, v_h)\in E(G)}x_h^2 \\
&=(1-\alpha)^2 s_k^+\sum\limits_{(v_k, v_h)\in E(G)}x_h^2,
\end{align*}
where $s_k^+= \frac{ \sum\limits_{(v_k, v_h)\in E(G)}b_h^2}{b_k^2}$.

Furthermore, we have
\begin{align*}
 \sum\limits_{k=1}^n(\lambda_\alpha(G)-\alpha d_k^{+})^2 x_k^2&\leq(1-\alpha)^2 \sum\limits_{k=1}^n\left(s_k^+\sum\limits_{(v_k, v_h)\in E(G)}x_h^2\right) \\
&=(1-\alpha)^2 \sum\limits_{k=1}^n\left(\sum\limits_{(v_h, v_k)\in E(G)}s_h^+\right)x_k^2,
\end{align*}
that is
\begin{equation}\label{eq:cb}
 \sum\limits_{k=1}^n\left[(\lambda_\alpha(G)-\alpha d_k^{+})^2-(1-\alpha)^2\sum\limits_{(v_h, v_k)\in E(G)}s_h^+ \right]x_k^2\leq0.
\end{equation}
Therefore, we can conclude that there exists some k such that
$$(\lambda_\alpha(G)-\alpha d_k^{+})^2-(1-\alpha)^2\sum\limits_{(v_h, v_k)\in E(G)}s_h^+\leq0,$$
that is
$$( \lambda_\alpha(G)-\alpha d_k^{+})^2\leq(1-\alpha)^2\sum\limits_{(v_h, v_k)\in E(G)}s_h^+.$$
Hence,  we have
$$ \lambda_\alpha(G)\leq\alpha d_k^{+}+(1-\alpha)\sqrt{\sum\limits_{(v_h, v_k)\in E(G)}s_h^+}.$$
Therefore, we get
 $$\lambda_\alpha(G)\leq \max\left\{\alpha d_i^{+} +(1-\alpha)\sqrt{\sum\limits_{(v_j,v_i) \in E(G)}s_j^+}
  : v_i \in V(G) \right\}.$$
 If the equality holds in \eqref{eq:ca}, then there exists some k, say $k=1$, such that
 $\lambda_\alpha(G)=\alpha d_1^{+} +(1-\alpha)\sqrt{\sum\limits_{(v_j,v_1) \in E(G)}s_j^+}$. Then, from \eqref{eq:cb}, we get
 $$ \sum\limits_{k=2}^n\left[(\lambda_\alpha(G)-\alpha d_k^{+})^2-(1-\alpha)^2\sum\limits_{(v_h, v_k)\in E(G)}s_h^+ \right]x_k^2\leq0.$$
 By a similar reasoning as above, we have
$$\lambda_\alpha(G)=\alpha d_k^{+} +(1-\alpha)\sqrt{\sum\limits_{(v_j,v_k) \in E(G)}s_j^+}, \ 1\leq k\leq n,$$
which means $\alpha d_i^{+} +(1-\alpha)\sqrt{\sum\limits_{(v_j,v_i) \in E(G)}s_j^+}$  ($1\leq i\leq n$) is a constant.
\end{proof}

\begin{corollary}\label{co:1}  \ Let $G$ be a strongly connected digraph with $n$ vertices.
Then
$$ \lambda_\alpha(G)\leq \max\left\{\alpha d_i^{+} +(1-\alpha)\sqrt{\sum\limits_{(v_j,v_i) \in E(G)}d_j^+} : v_i \in
V(G)\right\}$$
\end{corollary}

\begin{proof} \ Taking $b_i = 1$ in \eqref{eq:ca}, the
result follows.
\end{proof}

\section{The $A_\alpha$ energy of  a digraph}

In this section, we will define the $A_\alpha$ energy of a digraph and will also give a formula for the $A_\alpha$ energy in terms of the vertex outdegrees and
the number of closed directed walks of length 2. Moreover, we will characterize the extremal digraphs with maximum and minimum $A_\alpha$ energy among all directed trees
and unicyclic digraphs, respectively.

\begin{definition}\label{de:1} Let $A_\alpha(G)$ be the $A_\alpha$ matrix of a digraph $G$.  The $A_\alpha$ energy of $G$ is defined as
$$E^{A_\alpha}(G)=\sum\limits_{i=1}^n(\lambda^\alpha_i)^2,$$
where $n$ is the number of vertices and $\lambda^\alpha_i$ $(i=1,2,\ldots,n)$
are the eigenvalues of $A_\alpha(G)$.
\end{definition}

For example, let $\overrightarrow{C_4}$ be a directed cycle of length $4$, then
\begin{eqnarray*}
A_\alpha(\overrightarrow{C_4})&=&\left(\begin{array}{cccc}
 \alpha  & 1-\alpha & 0 & 0  \\
 0  & \alpha & 1-\alpha & 0 \\
 0 & 0 & \alpha & 1-\alpha \\
 1-\alpha & 0 & 0 & \alpha \\
\end{array}\right).
\end{eqnarray*}
The eigenvalues of $A_\alpha(\overrightarrow{C_4})$
are $2\alpha-1$, $1$, $\alpha+(1-\alpha)i$, $\alpha-(1-\alpha)i$. Hence the $A_\alpha$ energy of $\overrightarrow{C_4}$ is $4\alpha^2$.

In the following, we will give a formula for $E^{A_\alpha}(G)$.

\begin{theorem}\label{thm:5} Let $G$ be a digraph with the outdegree sequence $\{d_1^+, d_2^+, \ldots, d_n^+\}$, and $c_2$ be the number of closed directed walks of length 2. Then $E^{A_\alpha}(G)=\alpha^2\sum\limits_{i=1}^n(d_i^+)^2+(1-\alpha)^2c_2$.
\end{theorem}

\begin{proof} Let $\lambda^\alpha_i$ $(i=1,2,\ldots,n)$ be the eigenvalues of $A_\alpha(G)$. Since $E^{A_\alpha}(G)=\sum\limits_{i=1}^n(\lambda^\alpha_i)^2$, we have $E^{A_\alpha}(G)$ is equal to
the trace of the matrix $(A_\alpha(G))^2$. Therefore
\begin{align*}
E^{A_\alpha}(G)&=tr((A_\alpha(G))^2)\\
&=tr(\alpha^2(D(G))^2+\alpha(1-\alpha)D(G)A(G)\\
&\ \ \ \ +\alpha(1-\alpha)A(G)D(G)+(1-\alpha)^2(A(G))^2)\\
&=\alpha^2 tr((D(G))^2)+\alpha(1-\alpha)tr(D(G)A(G))\\
&\ \ \ \ +\alpha(1-\alpha)tr(A(G)D(G))+(1-\alpha)^2tr((A(G))^2).
\end{align*}
Note that $tr(D(G)A(G))=tr(A(G)D(G))=0$ and $tr((A(G))^2)=\sum\limits_{i=1}^n\sum\limits_{j=1}^na_{ij}a_{ji}=c_2$. Hence, we get
\begin{align*}
E^{A_\alpha}(G)&=\alpha^2 tr((D(G))^2)+(1-\alpha)^2tr((A(G))^2)\\
&=\alpha^2\sum\limits_{i=1}^n(d_i^+)^2+(1-\alpha)^2c_2.
\end{align*}
\end{proof}

\begin{corollary}\label{co:2} Let $H$ be a proper sundigraph of $G$, then
$$E^{A_\alpha}(H)<E^{A_\alpha}(G).$$
\end{corollary}

A connected digraph without symmetric arcs is called an intree if the outdegree of each vertex is at most one and its underlying undirected graph is a tree.  Obviously,  an intree
with $n$ vetices and $n-1$ arcs. Hence, it has exactly one vertex with outdegree zero, and we call such vertex the root of the intree.

\begin{theorem}\label{thm:6} Let $G$ be a connected digraph with $n$ vertices. Then $\alpha^2(n-1)\leq E^{A_\alpha}(G)\leq(\alpha^2n+1-2\alpha)n(n-1)$.
Moreover, $E^{A_\alpha}(G)=(\alpha^2n+1-2\alpha)n(n-1)$ if and only if $G$ is a symmetric complete digraph, and
$E^{A_\alpha}(G)=\alpha^2(n-1)$ if and only if $G$ is an intree.
\end{theorem}

\begin{proof} For the upper bound, from Theorem \ref{thm:5}, we get
\begin{align*}
E^{A_\alpha}(G)&=\alpha^2\sum\limits_{i=1}^n(d_i^+)^2+(1-\alpha)^2c_2\\
&\leq\alpha^2\sum\limits_{i=1}^n(d_i^+)^2+(1-\alpha)^2m\\
&=\alpha^2\sum\limits_{i=1}^n(d_i^+)^2+(1-\alpha)^2\sum\limits_{i=1}^nd_i^+\\
&\leq\alpha^2(n-1)^2n+(1-\alpha)^2(n-1)n\\
&=(\alpha^2n+1-2\alpha)(n-1)n,
\end{align*}
where $m$ is the number of arcs,  $c_2$ be the number of closed directed walks of length 2,  and the equality holds if and only if $c_2=m$ and $d_i^+=n-1$ for all $i$, which means $G$ is a symmetric complete digraph.

For the lower bound, let $G_0$ be a connected digraph which has the minimum $A_\alpha$ energy among all connected digraphs with $n$ vertices. Then we have the
following two claims.

{\bf Clam 1:} $c_2=0$, where $c_2$ is the number of closed directed  walks of length 2.

Otherwise, let $uvu$ be a closed directed walks of length 2. Then $G'=G_0-\{(u,v)\}$ is a proper connected subdigraph of $G_0$.  By Corollary \ref{co:1}, we have
$E^{A_\alpha}(G')<E^{A_\alpha}(G_0)$, which is a contradiction to the assumption that $G_0$ with the minimum $A_\alpha$ energy.

{\bf Clam 2:} The underlying undirected graph of $G$ is a tree.

Otherwise, there exists a cycle $C$ in the underlying undirected graph of $G$. Let $G'$ be a digraph obtained from $G_0$ by deleting any arc $e\in C$ from $G_0$, that is
$G'=G_0-e$. Then $G'$ is still connected. However, By Corollary \ref{co:1}, we have
$E^{A_\alpha}(G')<E^{A_\alpha}(G_0)$, which is also a contradiction.

Furthermore, by Claims 1 and 2, we get
$$E^{A_\alpha}(G)=\alpha^2\sum\limits_{i=1}^n(d_i^+)^2+(1-\alpha)^2c_2=\alpha^2\sum\limits_{i=1}^n(d_i^+)^2.$$
Since $\sum\limits_{i=1}^nd_i^+=n-1$ and $d_i^+\geq0$, there must exist one vertex with outdegree 0. Without loss of generality,
we assume $d_n^+=0$. Hence, $\sum\limits_{i=1}^{n-1}d_i^+=n-1$.
Therefore, by Cauchy-Schwarz inequality, we get
\begin{align*}
E^{A_\alpha}(G)&=\alpha^2\sum\limits_{i=1}^{n-1}(d_i^+)^2\\
&\geq \alpha^2\frac{(d_1^++d_2^++\cdots+d_{n-1}^+)^2}{n-1}\\
&=\alpha^2(n-1),
\end{align*}
and the equality holds if and only if $d_1^+=d_2^+=\cdots=d_{n-1}^+=1$ and $d_n^+=0$, which means $G_0$ is an intree.
\end{proof}

In the following, we will give another lower and upper bounds for the $A_\alpha$ energy of a connected digraph.

\begin{theorem}\label{thm:7} Let $G$ be connected digraph with $n$ vertices, $m$ arcs,
the maximum outdegree $\Delta^+$ and the minimum outdegree $\delta^+$. Then
$$\frac{\alpha^2m^2}{n}\leq E^{A_\alpha}(G)\leq[\alpha^2 (\Delta^++\delta^+)+(1-\alpha)^2]m-\alpha^2n\Delta^+\delta^+,$$
where the left equality holds if and only if $G$ is a digraph without symmetric arcs and each vertex has the same outdegree, the
right equality holds if and only if $G$ is a symmetric digraph and $d_i^+=\delta^+$ or $d_i^+=\Delta^+$ for each vertex $v_i$.
\end{theorem}

\begin{proof} Let $c_2$ denote the number of closed directed walks of length 2. Clearly, $0\leq c_2\leq m$, and $\sum\limits_{i=1}^{n}d_i^+=m$. Then we have
$$E^{A_\alpha}(G)=\alpha^2\sum\limits_{i=1}^n(d_i^+)^2+(1-\alpha)^2c_2\geq\alpha^2\sum\limits_{i=1}^n(d_i^+)^2\geq\alpha^2\frac{\left(\sum\limits_{i=1}^nd_i^+\right)^2}{n}=\frac{\alpha^2m^2}{n},$$
where the equality holds if and only if $c_2=0$ and $d_1^+=d_2^+=\cdots=d_n^+$, which means $G$ is a digraph without symmetric arcs and each vertex has the same outdegree.

On the other hand, since $\delta^+\leq d_i^+\leq\Delta^+$, for any $1\leq i\leq n$, we have
$$(d_i^+-\delta^+) (d_i^+-\Delta^+)\leq0.$$
Hence
\begin{align*}
E^{A_\alpha}(G)&=\alpha^2\sum\limits_{i=1}^n(d_i^+)^2+(1-\alpha)^2c_2\\
&\leq\alpha^2\sum\limits_{i=1}^n(d_i^+)^2+(1-\alpha)^2m\\
&\leq\alpha^2\sum\limits_{i=1}^n[(\Delta^++\delta^+)d_i^+-\Delta^+\delta^+]+(1-\alpha)^2m\\
&=\alpha^2[ (\Delta^++\delta^+)m-n\Delta^+\delta^+]+(1-\alpha)^2m\\
&=[\alpha^2 (\Delta^++\delta^+)+(1-\alpha)^2]m-\alpha^2n\Delta^+\delta^+,
\end{align*}
where the equality holds if and only if $c_2=m$ and $d_i^+=\delta^+$ or $d_i^+=\Delta^+$ for all $1\leq i\leq n$,
which means $G$ is a symmetric digraph and $d_i^+=\delta^+$ or $d_i^+=\Delta^+$ for each vertex $v_i$.
\end{proof}

Let $G$ be connected digraph,  $\bar{G}$ be the complement of $G$. In the following, we will give a lower and upper bounds for the $A_\alpha$ energy of  $\bar{G}$.

\begin{theorem}\label{thm:8} Let $G$ be a connected digraph with $n$ vertices, $m$ arcs, the maximum outdegree $\Delta^+$,
the minimum outdegree $\delta^+$, and $\bar{G}$ be the complement of $G$. Then
$$\alpha^2(n-1-\Delta^+)(n(n-1)-m)\leq E^{A_\alpha}(\bar{G})\leq[\alpha^2(n-1-\delta^+)+(1-\alpha)^2](n(n-1)-m),$$
where the left equality holds if and only if $G$ is a complete digraph and each vertex has the same outdegree, the
right equality holds if and only if $G$ is a symmetric digraph and each vertex has the same outdegree.
\end{theorem}

\begin{proof} Let $\overline{c_2}$ denote the number of closed directed walks of length 2 in the complement digraph $\bar{G}$, $\overline{d_i^+}$ denote the outdegree of vertex $v_i$
in the complement digraph $\bar{G}$. Since $\delta^+\leq d_i^+\leq\Delta^+$,  for any $1\leq i\leq n$, we have
 $$n-1-\Delta^+\leq \overline{d_i^+}\leq n-1-\delta^+,$$
 and
$$0\leq\overline{c_2}\leq n(n-1)-m.$$
Therefore
\begin{align*}
E^{A_\alpha}(\bar{G})&=\alpha^2\sum\limits_{i=1}^n(\overline{d_i^+})^2+(1-\alpha)^2\overline{c_2}\\
&\geq\alpha^2\sum\limits_{i=1}^n(\overline{d_i^+})^2\\
&\geq\alpha^2\sum\limits_{i=1}^n(n-1-\Delta^+)\overline{d_i^+}\\
&=\alpha^2(n-1-\Delta^+)(n(n-1)-m),
\end{align*}
where the equality holds if and only if $\overline{c_2}=0$ and $\overline{d_1^+}=\overline{d_2^+}=\cdots=\overline{d_n^+}=n-1-\Delta^+$,
which means $G$ is a complete digraph and each vertex has the same outdegree.

On the other hand,
\begin{align*}
E^{A_\alpha}(\bar{G})&=\alpha^2\sum\limits_{i=1}^n(\overline{d_i^+})^2+(1-\alpha)^2\overline{c_2}\\
&\leq\alpha^2\sum\limits_{i=1}^n(\overline{d_i^+})^2+(1-\alpha)^2(n(n-1)-m)\\
&\leq\alpha^2\sum\limits_{i=1}^n(n-1-\delta^+)\overline{d_i^+}+(1-\alpha)^2(n(n-1)-m)\\
&=\alpha^2(n-1-\delta^+)(n(n-1)-m)+(1-\alpha)^2(n(n-1)-m)\\
&=[\alpha^2(n-1-\delta^+)+(1-\alpha)^2](n(n-1)-m),
\end{align*}
where the equality holds if and only if $\overline{c_2}=n(n-1)-m$ and $\overline{d_i^+}=n-1-\delta^+$ for all $1\leq i\leq n$,
which means $G$ is a symmetric digraph and each vertex has the same outdegree.
\end{proof}

For a connected digraph $G$ with $n$ vertices and $m$ arcs, if $m=n-1$, then $G$ is called a directed tree.
A directed tree with $n$ vertices is called outstar $\overrightarrow{S_n}$ if the directed tree has one vertex with outdegreee $n-1$ and
the other vertices with outdegree 0, and we call the vertex with outdegreee $n-1$ the centre of the outstar. In the following, we will characterize the directed trees which have
the minimum and maximum $A_\alpha$ energy among
all directed trees with $n$ vertices, respectively.

\begin{theorem}\label{thm:9} Let $T$ be a directed tree with $n$ vertices.
Then $\alpha^2(n-1)\leq E^{A_\alpha}(G)\leq\alpha^2(n-1)^2$.
Moreover, $E^{A_\alpha}(T)=\alpha^2(n-1)^2$ if and only if $T$ is a outstar, and
$E^{A_\alpha}(G)=\alpha^2(n-1)$ if and only if $T$ is a intree.
\end{theorem}

\begin{proof} From Theorem \ref{thm:5}, we get
$$E^{A_\alpha}(T)=\alpha^2\sum\limits_{i=1}^n(d_i^+)^2\leq\alpha^2\left(\sum\limits_{i=1}^nd_i^+\right)^2=\alpha^2(n-1)^2,$$
where the equality holds if and only if  $\sum\limits_{i=1}^n(d_i^+)^2=\left(\sum\limits_{i=1}^nd_i^+\right)^2=(n-1)^2$,  which means
$T$ has one vertex with outdegreee $n-1$ and the other vertices with outdegree 0,  i.e.,  $T$ is a outstar.

On the other hand, by  Theorem \ref{thm:6}, we have $ E^{A_\alpha}(T)\geq\alpha^2(n-1)$, and the equality holds if and only if $T$ is a intree.
\end{proof}

Let $G$ be a connected digraph with $n$ vertices and $m$ arcs. If $m=n$, and $G$ contains a unique directed cycle,
then $G$ is called a unicyclic digraph. Let $C_a^{n-a}$ is a unicyclic digraph which contains a unique directed cycle
$\overrightarrow{C_a}$ and  an outstar $\overrightarrow{S_{n-a+1}}$ whose centre is on the $\overrightarrow{C_a}$, as shown in Figure \ref{Fig.1.}.
In particular, $C_2^{n-2}$ is a unicyclic digraph which contains a unique directed cycle
$\overrightarrow{C_2}$ and  an outstar $\overrightarrow{S_{n-1}}$  whose centre is on the $\overrightarrow{C_2}$, as shown in Figure \ref{Fig.1.}.
$D_a^{n-a}$ is a set of  unicyclic digraphs which contain a unique directed cycle
$\overrightarrow{C_a}$ and some intrees whose roots are on the $\overrightarrow{C_a}$. In the following, we will characterize the unicyclic digraphs which have
the minimum and maximum $A_\alpha$ energy among
all unicyclic digraphs with $n$ vertices, respectively.

\begin{figure}[H]
\begin{centering}
\includegraphics[scale=0.85]{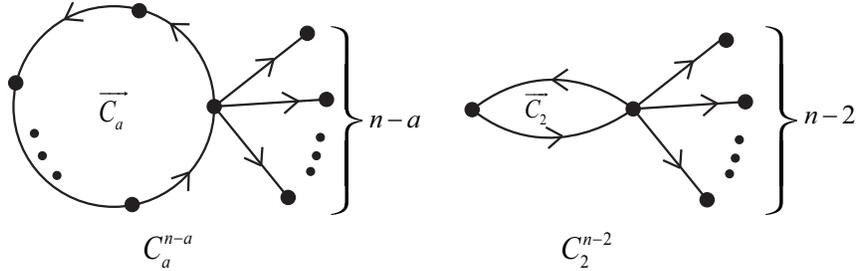}
\caption{The digraphs $C_a^{n-a}$ and $C_2^{n-2}.$}\label{Fig.1.}
\end{centering}
\end{figure}

\begin{theorem}\label{thm:10} Let $G$ be a unicyclic digraph with $n\geq2$ vertices. Then
$$\alpha^2n\leq E^{A_\alpha}(G)\leq\alpha^2(n^2-2n+2)+2(1-\alpha)^2,$$
where the left equality holds if and only if $G$ is $C_n$ or $G\in D_a^{n-a}$. And the right
equality holds if and only if $G$ is $C_2^{n-2}$, and $C_2^{n-2}$ as shown in Figure \ref{Fig.1.}.
\end{theorem}

\begin{proof} Let $\overrightarrow{C_a}$ be the unique directed cycle in the unicyclic digraph $G$ with $a\geq2$ vertices.
We use $T_1,T_2,\ldots,T_a$ to denote the directed trees obtained from $G$ by deleting all arcs on the $\overrightarrow{C_a}$. Let each
$T_i$ has $n_i$ vertices, and the outdegree of the vertex in $T_i$ is denoted by $d_{T_i^j}^+$, where $i=1,2,\ldots,a$ and $j=1,2,\ldots,n_i$. Without loss of generality, we assume that
$n_1\geq n_2\geq n_3\geq \cdots\geq n_a\geq1$. Obviously, $\sum\limits_{i=1}^a n_i=n$. In order to prove the upper bound, we consider the following two cases.

{\bf Case 1:} $a\geq3$. Then by Theorem \ref{thm:5}, we get
\begin{align*}
E^{A_\alpha}(G)&=\alpha^2\sum\limits_{i=1}^a\left(\sum\limits_{j=2}^{n_i}(d_{T_i^j}^+)^2+(d_{T_i^1}^++1)^2\right)\\
&=\alpha^2\sum\limits_{i=1}^a\left(\sum\limits_{j=1}^{n_i}(d_{T_i^j}^+)^2+2d_{T_i^1}^++1\right)\\
&\leq\alpha^2\sum\limits_{i=1}^a\left(\left(\sum\limits_{j=1}^{n_i}d_{T_i^j}^+\right)^2+2(n_i-1)+1\right)\\
&=\alpha^2\left(\sum\limits_{i=1}^a(n_i-1)^2+\sum\limits_{i=1}^a(2n_i-1)\right)\\
&\leq\alpha^2\left(\left(\sum\limits_{i=1}^a(n_i-1)\right)^2+2n-a\right)\\
&=\alpha^2\left(\left(n-a\right)^2+2n-a\right)\\
&=\alpha^2\left(\left(n-a+1\right)^2+a-1\right),
\end{align*}
where the equality holds if and only if $d_{T_i^1}^+=n_i-1$, $d_{T_i^2}^+=d_{T_i^3}^+=\cdots=d_{T_i^{n_i}}^+=0$ ($i=1,2,\ldots,a$), $n_1=n-a+1$, and
$n_2=n_3=\ldots=n_a=1$, which means $G$  is $C_a^{n-a}$, and $C_a^{n-a}$ as shown in Figure \ref{Fig.1.}.

Let $f(a)=(n-a+1)^2+a-1$. Then $f'(a)=-2(n-a)-1<0$, $f(a)$ is decreasing on $3 \leq a\leq n$. Thus we have
$f(a)=(n-a+1)^2+a-1\leq n^2-4n+6$. Furthermore, we have $E^{A_\alpha}(G)\leq\alpha^2(n^2-4n+6)$ with the equality if and only if
$G$  is $C_3^{n-3}$.

{\bf Case 2:} $a=2$. Then by Theorem \ref{thm:5}, we get
$$E^{A_\alpha}(G)=\alpha^2\sum\limits_{i=1}^n(d_i^+)^2+2(1-\alpha)^2.$$
Similar to the proof of Case 1, we have
$$E^{A_\alpha}(G)\leq\alpha^2\left(\left(n-2+1\right)^2+2-1\right)+2(1-\alpha)^2=\alpha^2(n^2-2n+2)+2(1-\alpha)^2,$$
with the equality if and only if
$G$  is $C_2^{n-2}$, and $C_2^{n-2}$ as shown in Figure \ref{Fig.1.}.

Combing the above two cases, since $n^2-2n+2\geq n^2-4n+6$, $\alpha^2(n^2-2n+2)+2(1-\alpha)^2>\alpha^2(n^2-4n+6)$, where $n\geq2$.
Hence,
$$E^{A_\alpha}(G)\leq\alpha^2(n^2-2n+2)+2(1-\alpha)^2,$$
where the equality holds if and only if $G$ is $C_2^{n-2}$, and $C_2^{n-2}$ as shown in Figure  \ref{Fig.1.}.

For the lower bound, we also consider the following two cases.

{\bf Case 1:} $a\geq3$. Then by Theorem \ref{thm:5}, we get
\begin{align*}
E^{A_\alpha}(G)&=\alpha^2\sum\limits_{i=1}^a\left(\sum\limits_{j=2}^{n_i}(d_{T_i^j}^+)^2+(d_{T_i^1}^++1)^2\right)\\
&=\alpha^2\sum\limits_{i=1}^a\left(\sum\limits_{j=1}^{n_i}(d_{T_i^j}^+)^2+2d_{T_i^1}^++1\right)\\
&\geq\alpha^2\sum\limits_{i=1}^a\left(\sum\limits_{j=1}^{n_i}d_{T_i^j}^++2d_{T_i^1}^++1\right)\\
&\geq\alpha^2\sum\limits_{i=1}^a(n_i-1+1)\\
&\geq\alpha^2\sum\limits_{i=1}^an_i=\alpha^2n,
\end{align*}
where the equality holds if and only if $d_{T_i^1}^+=0$ for all $i=1,2,\ldots,a$, and
$\sum\limits_{j=1}^{n_i}(d_{T_i^j}^+)^2=\sum\limits_{j=1}^{n_i}d_{T_i^j}^+=n_i-1$, which means $G$ is $\overrightarrow{C_n}$ or
$G\in D_a^{n-a}$.

{\bf Case 2:} $a=2$. Then by Theorem \ref{thm:5}, we get
$$E^{A_\alpha}(G)=\alpha^2\sum\limits_{i=1}^n(d_i^+)^2+2(1-\alpha)^2.$$
Similar to the proof of Case 1, we have
$$E^{A_\alpha}(G)\geq\alpha^2n+2(1-\alpha)^2.$$

Combing the above two cases, we get
$$E^{A_\alpha}(G)\geq\alpha^2n,$$
where the equality holds if and only if $G$ is $\overrightarrow{C_n}$ or
$G\in D_a^{n-a}$.
\end{proof}

\end{document}